\pdfoutput=1
\RequirePackage{ifpdf}
\ifpdf 
\documentclass[pdftex]{sigma}
\else
\documentclass{sigma}
\fi

\numberwithin{equation}{section}

\newtheorem{Theorem}{Theorem}[section]
\newtheorem*{Theorem*}{Theorem}

\newtheorem{Proposition}[Theorem]{Proposition}
\newtheorem{problem}[Theorem]{Problem}
\newtheorem{spec}[Theorem]{Speculation}
 { \theoremstyle{definition}

}

\newcommand\qbin[3]{{\left[\begin{matrix} #1 \\ #2 \end{matrix} \right]_{#3}}}

\begin{document}

\renewcommand{\thefootnote}{}

\renewcommand{\PaperNumber}{049}

\FirstPageHeading

\ShortArticleName{The Rogers--Ramanujan Identities and Cauchy's Identity}

\ArticleName{The Rogers--Ramanujan Identities\\ and Cauchy's Identity\footnote{This paper is a~contribution to the Special Issue on Basic Hypergeometric Series Associated with Root Systems and Applications in honor of Stephen C.~Milne's 75th birthday. The~full collection is available at \href{https://www.emis.de/journals/SIGMA/Milne.html}{https://www.emis.de/journals/SIGMA/Milne.html}}}

\Author{Dennis STANTON}

\AuthorNameForHeading{D.~Stanton}

\Address{School of Mathematics, University of Minnesota, Minneapolis, Minnesota 55455, USA}
\Email{\href{mailto:stanton@math.umn.edu}{stanton@math.umn.edu}}
\URLaddress{\url{https://www-users.cse.umn.edu/~stant001/}}

\ArticleDates{Received February 26, 2025, in final form June 26, 2025; Published online July 01, 2025}

\Abstract{The Rogers--Ramanujan identities are investigated using the Cauchy identity for Schur functions.}


\Classification{05A17; 05A19}

\begin{flushright}
\begin{minipage}{46mm}
\it Dedicated to Stephen Milne\\ for his 75th birthday
\end{minipage}
\end{flushright}

\renewcommand{\thefootnote}{\arabic{footnote}}
\setcounter{footnote}{0}

\section{Introduction}

Two of Steve Milne's most noteworthy works are on the Rogers--Ramanujan identities
 (see \cite{And2})
 \begin{gather}
 \label{RR}
 \sum_{n=0}^\infty \frac{q^{n^2}}{(q;q)_n}= \frac{1}{\bigl(q^1;q^5\bigr)_\infty \bigl(q^4;q^5\bigr)_\infty},
 \qquad \sum_{n=0}^\infty \frac{q^{n^2+n}}{(q;q)_n}= \frac{1}{\bigl(q^2;q^5\bigr)_\infty \bigl(q^3;q^5\bigr)_\infty}.
 \end{gather}
 With J.~Lepowsky he proved \eqref{RR} algebraically (see \cite{LM1,LM2}).
 The involution principle, with A.~Garsia \cite{GM}, gave an indirect bijection for MacMahon's
 combinatorial interpretation of the identities.

Stembridge \cite{JS} used symmetric function identities via Hall--Littlewood polynomials
to prove and generalize the Rogers--Ramanujan identities.
This was continued by Jouhet--Zeng~\cite{JZ} and S.~Ole Warnaar~\cite{OW1}.
A vast generalization to the Rogers--Ramanujan identities, corresponding to affine
Lie algebras, was given in~\cite{GOW,RW}. Here the appropriate Hall--Littlewood
polynomials are specializations of the Macdonald--Koornwinder polynomials.

The purpose of this note is explore a naive approach using the Cauchy
identity for Schur functions. What would be required for a explicit bijective proof via the
Cauchy identity is discussed in Section \ref{sec2}. Some related identities and a speculation are
given in Sections \ref{sec3} and \ref{sec4}, while Section \ref{sec5} has two remarks.

All symmetric function facts can be found in Macdonald's book~\cite{Mac}.

\section{A proposal for a bijection}\label{sec2}

MacMahon's combinatorial interpretation of \eqref{RR} uses integer partitions.
\begin{Proposition}
The first Rogers--Ramanujan identity is equivalent to the following two sets of integer
partitions being equinumerous for any $n$:
\begin{enumerate}\itemsep=0pt
\item[$(1)$] integer partitions of $n$ into parts congruent to $1$ or $4$ modulo $5$,
\item[$(2)$] integer partitions of $n$ whose parts differ by at least $2$.
\end{enumerate}
\end{Proposition}

There is no known direct bijection between these two finite sets of partitions.
There is a~similar statement for the second Rogers--Ramanujan identity, also with
an unknown bijection.

In this paper, we use Schur functions, $s_\lambda(x_1,\dots, x_n,\dots),$
which are symmetric functions in variables
$x_1,x_2,\dots$ indexed by integer partitions $\lambda.$
A Schur function indexed by $\lambda$ is the generating function of all column strict
tableaux $P$ of shape $\lambda$. For example, if $\lambda=(4,2,1)$, one such~$P$~is\looseness=-1
\begin{gather*}
\begin{matrix}
&1&1&3&6\\
P=&3&3&&\\
&5&&&
\end{matrix}
\end{gather*}
whose weight is $x_1^2x_3^3x_5^1x_6^1.$ In this paper, the weights are
always powers of $q$, so the weight of a~column strict tableaux $P$ is
$q^N$, where $N$ is the sum of the entries of $P$. Suppose the number of variables
is finite, for example $R$ variables.
There are only $R$ possible choices for entries in the first column, so if the
indexing partition $\lambda$ has more than $R$ parts, the Schur function is zero.
We will later use $R=2$ case so that $\lambda$ has at most 2 rows.
The Schur function indexed by the empty partition is 1.

The Cauchy identity for Schur functions $s_\lambda(x_1,\dots, x_n,\dots)$
provides a start for a Rogers--Ramanujan bijection.
The Cauchy identity is
\begin{gather}
\label{Cauchy}
\sum_{\lambda} s_\lambda(x_1,\dots, x_n,\dots) s_\lambda(y_1,\dots, y_m,\dots)=
\prod_{i,j} (1-x_iy_j)^{-1}.
\end{gather}
Moreover, it is known that the Robinson--Schensted--Knuth correspondence
 is a direct bijection for~\eqref{Cauchy}, see \cite[Chapter~11.3]{And2}.

Choose
$
(x_1,\dots, x_n,\dots)= \bigl(1,q^5,q^{10}, q^{15},\dots\bigr)$, $ (y_1,y_2)=\bigl(q^1,q^4\bigr)$
so that the right side of~\eqref{Cauchy} is the product side of the first
Rogers--Ramanujan identity
\smash{$
\frac{1}{(q;q^5)_\infty (q^4;q^5)_\infty}$},
while the left side is restricted to partitions with at most two rows
\[
\sum_{\lambda{\text { at most 2 rows}} }
s_\lambda\bigl(1,q^5,q^{10}, q^{15},\dots\bigr) s_\lambda\bigl(q^1,q^4\bigr).
\]

\begin{Proposition}\label{RSK}
The Robinson--Schensted--Knuth correspondence provides a direct
bijection between
\begin{enumerate}\itemsep=0pt
\item[$(1)$] integer partitions of $n$ into parts congruent to $1$ or $4$ modulo $5$,
\item[$(2)$] pairs of column strict tableaux $(P,Q)$ of the same shape with at most two rows,
$P$ having entries congruent to $0$ modulo $5$, $Q$ having entries $1$ or $4$,
whose sum of entries is $n$.
\end{enumerate}
\end{Proposition}

Proposition~\ref{RSK} offers some advantages and disadvantages for a bijection.
On the plus side, it changes the problem to a problem on tableaux, for which there is a
well developed machinery of bijections. These more refined objects may be easier to sort than
integer partitions. Conversely, the simple answer required, partitions whose parts
differ by at least two, may not be apparent from this detailed view.
The $\lambda=\varnothing$ term corresponds to the $n=0$ term
in the sum side of the first Rogers--Ramanujan identity, namely 1.
Here are the pairs of column strict tableaux $(P,Q)$ which correspond to the $n=1$
term in the sum side of the first Rogers--Ramanujan identity
\begin{gather*}
\frac{q^1}{1-q}=\frac{q^1+q^2+q^3+q^4+q^5}{1-q^5}.
\end{gather*}
If $\lambda=1$, the possible choices for $(P,Q)$ are
$(x,1)$, $(x,4)$, where $x$ is a multiple $5$.
Their generating function is
\smash{$
\frac{q^1+q^4}{1-q^5}$}.
For $\lambda=2$, we may take
$((0,x), (1,1))$, $((0,x),(1,4))$ where~$x$ is a multiple $5$,
whose generating function is
\smash{$\frac{q^2+q^5}{1-q^5}$}.
For the remaining term, we take $\lambda=3$,
$((0,0,x), (1,1,1))$, where~$x$ is a multiple~$5$.
whose generating function is \smash{$\frac{q^3}{1-q^5}$}.

\begin{table}[ht]\centering\renewcommand{\arraystretch}{1.15}
\caption{Column strict pairs $(P,Q)$ for $n=1$.}\label{table:first}
\vspace{1mm}

\begin{tabular}{c c c c}
\hline
$\lambda $& $(P,Q)$ &generating function\\
\hline
1&($(x)$, $(1)$ or $(4)$) & $\bigl(q^1+q^4\bigr)/\bigl(1-q^5\bigr)$&\\
2&($(0,x)$, $(1,1)$ or $(1,4)$) &$\bigl(q^2+q^5\bigr)/\bigl(1-q^5\bigr)$\\
3&($(0,0,x)$, $(1,1,1)$) & $q^3/\bigl(1-q^5\bigr)$\\
\end{tabular}
\end{table}

The $n=2$ term on the sum side is
\begin{gather*}
\frac{q^4}{(1-q)\bigl(1-q^2\bigr)}\\
\qquad=
\frac{q^4 + q^5 + 2q^6 + 2 q^7 + 3 q^8 + 2 q^9 + 3 q^{10} + 2 q^{11} +
 3 q^{12} + 2 q^{13} + 2 q^{14} + q^{15} + q^{16}}{\bigl(1-q^5\bigr)\bigl(1-q^{10}\bigr)}.
\end{gather*}

We list, in Table~\ref{table:second}, 25 classes of pairs $(P,Q)$ which correspond to these 25 numerator terms.
Each class has a generating
function of $q^A/\bigl(1-q^5\bigr)\bigl(1-q^{10}\bigr)$, for an $A$ between $4$ and $16$. The denominator factors occur
because the generating function for partitions with at most 2 parts, each part a multiple of 5, is
$1/\bigl(1-q^5\bigr)\bigl(1-q^{10}\bigr).$

\begin{table}[ht]\centering\renewcommand{\arraystretch}{1.15}
\caption{Column strict pairs $(P,Q)$ for $n=2$.}\label{table:second}\vspace{1mm}

\begin{tabular}{@{\,}c@{\,}c@{\,}c@{\,}c@{\,}}
\hline
$\lambda $& $(P,Q)$ &$A$\\
\hline
2&($(y,x)$, $(1,1)$ or $(1,4)$) $y\ge 5$ & 12, 15\\
2&($(y,x)$, $(4,4)$) & 8\\
3&($(0,y,x)$, $(1,1,4)$) or $(1,4,4)$ or $(4,4,4)$ &6, 9, 12\\
4&($(0,0,y,x)$, $(1,1,1,1)$ or $(1,1,1,4)$ or $(1,1,4,4)$ or $(1,4,4,4)$ & 4, 7, 10, 13, 16\\
 & or $(4,4,4,4)$) &\\
5&($(0,0,0,y,x)$, $(1,1,1,1,1)$ or $(1,1,1,1,4)$ or $(1,1,1,4,4)$ or $(1,1,4,4,4)$) &5, 8, 11, 14\\
6&($(0,0,0,0,y,x)$, $(1,1,1,1,1,1)$ or $(1,1,1,1,1,4)$ or $(1,1,1,1,4,4)$) &6, 9, 12\\
7&($(0,0,0,0,0,y,x)$, $(1,1,1,1,1,1,1)$ or $(1,1,1,1,1,1,4)$ &7, 10, 13\\
 & or $(1,1,1,1,1,4,4)$)&\\
8& ($(0,0,0,0,0,0,y,x)$, $(1,1,1,1,1,1,1,1)$ or $(1,1,1,1,1,1,1,4)$ &8, 11, 14\\
& or $(1,1,1,1,1,1,4,4)$) &\\
$(1,1)$&(transpose$(y,x)$, $x> y$, transpose$(1,4)$)&10
\end{tabular}
\end{table}

In general, we want to obtain the $n$-th term in the Rogers--Ramanujan sum
\[
\sum_{n=0}^\infty \frac{q^{n^2}}{(q;q)_n}= 1+\frac{q}{1-q}+\frac{q^4}{(1-q)\bigl(1-q^2\bigr)}+\cdots =\sum_{n=0}^\infty \frac{q^{n^2}}{\bigl(q^5;q^5\bigr)_n} \prod_{j=1}^n \sum_{p=0}^4 q^{jp}.
\]

\begin{problem}\label{mainprob}
Can one choose pairs of column strict tableaux $(P,Q)$ of the same shape such that\looseness=-1
\begin{enumerate}\itemsep=0pt
\item[$(1)$] the entries of $P$ are multiples of 5,
\item[$(2)$] the entries of $Q$ are $1$ and $4$,
\item[$(3)$] and whose generating function is
\[
F_n(q)= q^{n^2}\left( \prod_{j=1}^n \sum_{p=0}^4 q^{jp}\right)/\bigl(q^5;q^5\bigr)_n?
\]
\end{enumerate}
\end{problem}
Solving Problem~\ref{mainprob} gives a Rogers--Ramanujan bijection
when combined with Proposition~\ref{RSK}. The pair of column strict tableaux
$(P,Q)$ correspond to an integer partition $\mu$ whose parts are $1$ or $4$ modulo $5$.
But they also correspond to an integer partition $\lambda$ whose parts differ by two. For example,
the $n=2$ term has the generating function
\begin{gather}
\label{mod5}
q^4 \left(\frac{1+q+q^2+q^3+q^4}{1-q^5} \right)
\left( \frac{1+q^2+q^4+q^6+q^8}{1-q^{10}}\right).
\end{gather}
This means, after subtracting $1$ from the second part of $\lambda$ and $3$ from the first part of
$\lambda$, the resulting columns have length $1$ or $2$. The 5 terms in the numerator
factors of \eqref{mod5} are the mod~5 values of the multiplicities of $1$ and $2.$
It will take substantially more insight to resolve Problem~\ref{mainprob} for an arbitrary $n$.

\section{Formulas}\label{sec3}

For clarity, here are the explicit generating functions of the Schur functions as products.
These follow from the principle specialization of Schur functions, the hook-content formula.
\begin{Proposition}
Let $\lambda=(a+b,a)$. Then
\begin{gather*}
s_\lambda\bigl(1,q^5,q^{10}, q^{15},\dots\bigr) = \frac{q^{5a}}{\bigl(q^5;q^5\bigr)_a \bigl(q^5;q^5\bigr)_b \bigl(q^{5(b+2)};q^5\bigr)_a},\qquad
s_\lambda\bigl(q^1,q^4\bigr)= q^{5a+b}\sum_{k=0}^b q^{3k}.
\end{gather*}
\end{Proposition}

There is a weighted version using two new parameters $x$ and $y$.
\begin{Theorem}
\label{xyRR}
Choosing $y_1=xq^1$, $y_2=yq^4$, we have
\[
\frac{1}{\bigl(xq;q^5\bigr)_\infty \bigl(yq^4;q^5\bigr)_\infty}=\sum_{a,b\ge 0}
\frac{q^{5a}}{\bigl(q^5;q^5\bigr)_a \bigl(q^5;q^5\bigr)_b \bigl(q^{5(b+2)};q^5\bigr)_a}
x^ay^a q^{5a+b}\sum_{k=0}^b x^{b-k}y^kq^{3k}.
\]
\end{Theorem}

Theorem~\ref{xyRR} independently follows from the finite identity
\[
\sum_{a=0}^M \qbin{N}{a}{q} \bigl(q^a-q^{N-a}\bigr)= \frac{(q;q)_N}{(q;q)_M (q;q)_{N-M-1}}
\qquad \text{for $0\le M\le N-1$}.
\]
Finally, a simple subclass of $(P,Q)$ has a product formula.
A proof of a more general result is given in Theorem~\ref{genthm}.
\begin{Proposition} We have
\begin{gather}
\sum_{\lambda{\text { at most $1$ row}} }
s_\lambda\bigl(1,q^5,q^{10}, q^{15},\dots\bigr) s_\lambda\bigl(q^1,q^4\bigr)= \frac{1}{1-q^3}
\biggl( \frac{1}{\bigl(q^1;q^5\bigr)_\infty}-\frac{q^3}{\bigl(q^4;q^5\bigr)_\infty}\biggr).\label{RR5-1row}
\end{gather}
\end{Proposition}

\section[Rogers--Ramanujan mod 2k+3]{Rogers--Ramanujan mod $\boldsymbol{ 2k+3}$}\label{sec4}

The same steps as in Section~\ref{sec2} can be done for higher
moduli $2k+3$, the integer partitions whose parts avoid
$\pm i$ and $0$ mod $2k+3,$ $1\le i\le 2k+2.$
Set
\begin{gather}
(x_1,\dots, x_n,\dots) = \bigl(1,q^{2k+3},q^{2(2k+3)}, q^{3(2k+3)},\dots\bigr),\nonumber\\
 (y_1,\dots, y_{2k}) =\bigl(q^1,\dots, q^{2k+2}\bigr)\qquad
{\text{with $q^i$ and $q^{2k+3-i}$ deleted}},\label{defparams}
\end{gather}
so that
\begin{gather*}
\sum_{\lambda{\text { at most $2k$ rows}} }
s_\lambda(x_1,\dots, x_n,\dots) s_\lambda (y_1,\dots, y_{2k})=
\prod_{\underset{j\not\equiv \pm i,0\bmod 2k+3}{j=1}}^\infty \bigl(1-q^j\bigr)^{-1}.
\end{gather*}
The product side of the Rogers--Ramanujan identities \eqref{RR} are the
$k=1$ and $i=2,1$ special cases.

There is always a version of the subclass formula \eqref{RR5-1row} as a sum of
infinite products using~\eqref{defparams}.

\begin{Theorem}
\label{genthm}
Let $k\ge 1$, $1\le i\le 2k+2$ and
$
(y_1,\dots, y_{2k}) =\bigl(q^1,\dots, q^{2k+2}\bigr)$,
with $q^i$ and~$q^{2k+3-i}$ deleted.
Then
\begin{gather*}
\sum_{\lambda{\text { at most $1$ row}} }
s_\lambda\bigl(1,q^{2k+3},q^{2(2k+3)}, q^{3(2k+3)},\dots\bigr) s_\lambda(y_1,\dots, y_{2k})=\sum_{\underset{p\neq i,2k+3-i}{p=1}}^{2k+2} \frac{A_p}{\bigl(q^p;q^{2k+3}\bigr)_\infty} ,
\end{gather*}
where
\begin{gather*}
A_p= \prod_{\underset{j\neq p, i,2k+3-i}{j=1}}^{2k+2}\frac{1}{1-q^{-p+j}}.
\end{gather*}
\end{Theorem}

\begin{proof} If $\lambda=N$ has a single part we have
\begin{gather*}
 s_\lambda\bigl(1,q^{2k+3},q^{2(2k+3)}, q^{3(2k+3)},\dots\bigr) =\frac{1}{\bigl(q^{2k+3};q^{2k+3}\bigr)_N},\\
s_\lambda(y_1,\dots, y_{2k})= {\text{ the coefficient of $t^N$ in }}
\prod_{\underset{j\neq i,2k+3-i}{j=1}}^{2k+2} \bigl(1-tq^j\bigr)^{-1}.
\end{gather*}

By partial fractions on $t$ we see that the $A_p$ satisfy
\[
\prod_{\underset{j\neq i,2k+3-i}{j=1}}^{2k+2} \bigl(1-tq^j\bigr)^{-1} =
\sum_{\underset {p\neq i,2k+3-i}{p=1}}^{2k+2} A_p(1-tq^p)^{-1},
\]
so that
\[
s_\lambda(y_1,\dots, y_{2k}) = \sum_{\underset{p\neq i,2k+3-i}{p=1}}^{2k+2}
A_p q^{pN}.
\]
We then use
\begin{gather*}
\sum_{N=0}^\infty \frac{q^{pN}}{\bigl(q^{2k+3};q^{2k+3}\bigr)_N}= \frac{1}{\bigl(q^p;q^{2k+3}\bigr)_\infty}.
\end{gather*}
to complete the proof.
\end{proof}

Note that Theorem~\ref{genthm} for $k=1$ and $i=2$ is \eqref{RR5-1row}.

\begin{spec}\label{spec1}
The $1$ row and at most $2k$ rows cases are sums of products
in the Rogers--Ramanujan infinite product. Perhaps this works for
any number of rows $R\le 2k.$ Is
\begin{gather*}
\sum_{\lambda{\text { at most $R$ rows}} }
s_\lambda(x_1,\dots, x_n,\dots) s_\lambda (y_1,\dots, y_{2k})
\end{gather*}
a sum of a product of $R$ infinite products, each of the form,
$
1/\bigl(q^j;q^{2k+3}\bigr)_\infty$, $j\not\equiv \pm i$, $0 \bmod {2k+3}$,
with coefficients which are rational functions in $q$?
\end{spec}

Note that Speculation~\ref{spec1} holds for $R=1$ and $R=2k.$

\section{Other symmetric function Cauchy identities}\label{sec5}

Michael Schlosser has pointed out that the dual Cauchy identity
\[
\sum_{\lambda} s_\lambda(x_1,\dots, x_n) s_{\lambda'}(y_1,\dots, y_m)=
\prod_{i=1}^n\prod_{j=1}^m (1+x_iy_j)
\]
can be similarly used with $m=2$ and
$
(x_1,x_2,\dots, x_n)=\bigl(1,q^3,\dots ,q^{3(n-1)}\bigr)$, $ (y_1,y_2)=\bigl(-q,-q^2\bigr)$
to obtain the Borwein product $\bigl(q^1;q^3\bigr)_n \bigl(q^2,q^3\bigr)_n.$
Again column strict tableaux could be used to approach
that problem, see~\cite{And}.

There are other Cauchy identities. If $x$ and $y$ are arbitrary sets of variables,
the Macdonald polynomials satisfy
\begin{gather}\label{Cauchy-Mac}
\sum_{\lambda} P_\lambda(x;q,t) Q_\lambda(y;q,t)=
\prod_{i,j} \frac{(tx_iy_j;q)_\infty}{(x_iy_j;q)_\infty}.
\end{gather}

Special cases of \eqref{Cauchy-Mac}, restricted by rows, have been
extensively used by Rains and S. Ole Warnaar \cite{RW}.

\subsection*{Acknowledgements} The author would like to thank the anonymous referees whose
suggestions substantially improved this paper.

\pdfbookmark[1]{References}{ref}
\LastPageEnding

\end{document}